\documentclass[12pt]{article}

\usepackage{amssymb, amsthm}

\newtheorem{theorem}{Theorem}[section]

\newtheorem{lemma}[theorem]{Lemma}

\def\irr#1{{\rm  Irr}(#1)}
\def\ibr#1{{\rm IBr}(#1)}

\def\phi{\varphi}

\newcommand{\Bpi}[1]{{\rm B}_\pi (#1)}
\newcommand{\Ipi}[1]{{\rm I}_\pi (#1)}

\title{Upper and lower bounds on Chillag table sums}
\author{Xiaoyou Chen \\ College of Science, Henan University of Technology, \\ Zhengzhou 450001, China \\ e-mail: cxymathematics@hotmail.com \\ \\ Mark L. Lewis and Hung P. Tong-Viet \\ Department of Mathematical Sciences, \\ Kent State University, Kent, OH 44242 \\ e-mail: lewis@math.kent.edu and htongvie@math.kent.edu}

\date{February 6, 2017}

\begin{document}

\maketitle

\begin{abstract}
Chillag has showed that there is a single generalization showing that the sums of ordinary character tables, Brauer character, and projective indecomposable characters are positive integers.  We show that Chillag's construction also applies to Isaacs' $\pi$-partial characters.  We show that if an extra condition is assumed, then we can obtain upper and lower bounds on the Chillag's table sums.  We will demonstrate that this condition holds for Brauer characters, $\pi$-partial characters, and projective indecomposable characters, and so we obtain upper and lower bounds for the table sums in those cases.  We also obtain results regarding the sums of rows and columns in these tables.

MSC[2010]: 20C15, 20C20

Keywords: Character table, Brauer characters, Projective indecomposable character
\end{abstract}

\section{Introduction}

All groups will be finite.  Let $G$ be a group and define $s (G)$ by $s (G) = \sum_{i=1}^n \sum_{j=1}^n \chi_i (g_j)$ where $\chi_1, \dots, \chi_n$ are the irreducible characters of $G$ and $g_1, \dots, g_n$ are representatives of the conjugacy classes of $G$.  Note that $s (G)$ is the sum of the entries in the character table of $G$.  In \cite{Solomon}, L. Solomon showed that $s(G)$ is a positive integer.  Furthermore, Solomon showed that $|A| \le s (G) \le |G|$ where $A$ is an abelian subgroup of $G$ of maximal order and he determined necessary and sufficient conditions for $s(G)$ to equal either $|G|$ or $|A|$.

Chillag showed that linear algebra could be used to show that $s (G)$ is a positive integer (see \cite{arcata}), and he then used these techniques to obtain other information regarding the ordinary characters of $G$.  Furthermore, in \cite{chilbrau}, Chillag used a similar linear algebra argument to obtain information about the Brauer characters of $G$, and Proposition 2.5 (i) of \cite{chilbrau} can be used to show that the sum of the elements of the Brauer character table of $G$ for a prime $p$ is a positive integer.  Finally, in \cite{chilgen}, Chillag gave definitions that generalized both ordinary characters and Brauer characters, and  formalized the linear algebra for this generalized situation.  He also showed that his arguments could be applied to conjugacy classes and projective indecomposable characters.

We will show that Chillag's work can be applied to Isaacs' $\pi$-partial characters of $\pi$-separable groups to show that the sum of the entries in the ``$\pi$-partial character table'' of $G$ is a positive integer.  Also, we will show that by adding an additional assumption to Chillag's definition, we can obtain an upper and a lower bound on the sums of the various ``tables.''  We will apply this bound to obtain upper and lower bounds for the table sums for Brauer characters, projective indecomposable characters, and $\pi$-partial characters.

The key point from Chillag's work is that the ``row sums,'' i.e. the sum of the elements in a row of such a table, are nonnegative integers.  It makes sense to ask what can be said about the column sums.  In \cite{chendu}, they showed that the column sums for the Brauer characters of $p$-solvable groups are integers.  We will show that if we add a different additional hypothesis, we can show that the column sums in Chillag's tables are integers.  We will apply this to see that column sums for ordinary characters and $\pi$-partial characters are integers.  
We will see that Brauer characters for groups that are not $p$-solvable do not necessarily satisfy our additional hypothesis, and we will give examples where the column sums are not integers.

\section{Extending Chillag's work}

In this section, we summarize Chillag's definitions and results from \cite{chilgen}.  We begin by letter $F$ be a subfield of the real numbers.  We say $A$ is a {\it SFCA} over $F$ if $A$ is a semisimple, finite-dimensional commutative $F$-algebra.  Let $B = \{ b_1, \dots, b_n \}$ be a basis for $A$ over $F$, and we let
$$
b_i b_j = \sum_{k=1}^n \alpha_{ijk} b_k.
$$
We say that the $\alpha_{ijk} \in F$ are the {\it structure constants} for $B$.  If all the structure constants are nonnegative real numbers, then we say that $B$ is a {\it nonnegative basis} for $A$.  For $a \in A$, we write
$$
ab_i = \sum_{j=1}^n m_{ij} (a) b_j,
$$
where $m_{ij} (a) \in F$ are uniquely determined.  We set $M (a,B)$ to be the matrix whose $(i,j)$-entry is $m_{ij} (a)$.  Notice that $m_{jk} (b_i) = a_{ijk}$, and so it follows that $B$ is a nonnegative basis for $A$ if and only if the entries $M (b_i,B)$ are nonnegative real numbers for all $i$.

For every element $a \in A$, we can write $a = \sum_{i=1}^n \alpha_i (a) b_i$.  We say that $a$ is {\it nonnegative} if all the $\alpha_i (a)$ are nonnegative real numbers.  When $B$ is a nonnegative basis and $a$ is a nonnegative element of $A$, it follows that the entries of $M(a,B)$ are all nonnegative, and so, the Perron-Frobenius Theorem can be applied to say that $M (a,B)$ has a positive real eigenvalue that equals the spectral radius of $M (a,B)$, and following Chillag, we call this the leading eigenvalue of $M (a,B)$ and we denote it by $\rho (a)$.  We then say that $B$ is a {\it right-positive basis} for $A$ if $B$ is a nonnegative basis of $A$ and for every nonnegative element $a \in A$, the matrices $M (a,B)$ have eigenvectors for $\rho (a)$ that contain positive real numbers.

For each element $a \in A$, we write $a(1) = \rho (a), \dots, a(n)$ for eigenvalues of $M(a,B)$.  Note that we may need to go to an extension field of $F$ to obtain these eigenvalues.  Chillag shows in Definition 1.3 and Theorem 1.4 of \cite{chilgen} when $B$ is a right-positive basis for $A$ that the ordering can be done uniformly for all $a \in A$, and in this uniform ordering, the eigenvalue $a (1) = \rho (a)$ has the eigenvector $(b_1 (1), \dots, b_n (1))$ for $M (a,B)$.

Following Chillag, we define $X(A,B)$ to be the $n \times n$ matrix whose $i,j$th-entry is $b_i (j)$, and as Chillag does, we say that $X(A,B)$ is the {\it table} for $(A,B)$.  Note that the entries in $X (A,B)$ lie in some field extension of $F$.  We set $s (A,B) = \sum_{i=1}^n \sum_{j=1}^n b_i(j)$ to be the {\it table sum} of $X (A,B)$ and note that $s (A,B)$ is the sum of the entries in $X (A,B)$.  For each $i$, we set $s_i (A,B) = s (b_i,B) = \sum_{j=1}^n  b_i (j)$ to be the $i$th row sum, i.e. the sum of the entries in the $i$th row.  Observe the $s (A,B) = \sum_{i=1}^n s_i (A,B)$.

In Theorem 1.4(2) of \cite{chilgen}, Chillag proved when $B$ is a right-positive basis that
$$
(X(A,B))^{-1} M(A,B) X(A,B) = {\rm diag} (a(1), \dots, a (n)).
$$
This implies that $s_i (A,B) = {\rm Tr} (M(b_i,B))$.  Hence, $s_i (A,B)$ will be a nonnegative real number for every $i$, and thus, $s (A,B)$ is a nonnegative real number.  (This is the content of Proposition 2.2 (b) in \cite{chilgen}.)  Notice that if the entries of the $M (b_i,B)$'s are integers (in fact, we only need the entries on the diagonals, i.e. the values $m_{jj} (b_i)$  to be integers), then the $s_i (A,B)$ will be integers, and $s (A,B)$ is an integer.  Note that if there exists an $i$ so that $b_i (j) = 1$ for all $j$, then $s_i (A,B) = n$ and $n \le s (A,B)$, so in this case, we know that $s (A,B)$ is positive.  We now show that if we add an additional hypothesis, we can obtain an upper bound and a different lower bound.

\begin{theorem} \label{bounds}
Let $A$ be an SFCA with right-positive basis $B$, and let $X = X (A,B)$ and $s = s (A,B)$ be as above.  Suppose that there exist real valued vectors $u$ and $v$ so that the entries in $u$ are positive, the entries in $v$ are nonnegative, and $uA = v$.  If $u_{min}$ is the minimum value of the entries in $u$, $u_{max}$ is the maximum value of the entries in $u$, and $v_{sum}$ is the sum of the entries in $v$, then
$$
\frac {v_{sum}}{u_{max}} \le s \le \frac {v_{sum}}{u_{min}}.
$$
\end{theorem}

\begin{proof}
Let $u = \langle u_1, \dots, u_n \rangle$ and $v = \langle v_1, \dots, v_n \rangle$.  Then $v_j = \sum_{i=1}^n u_i b_i (j)$, and so,
$$
v_{sum} = \sum_{j=1}^n v_j = \sum_{j=1}^n \sum_{i=1}^n u_i b_i (j) = \sum_{i=1}^n u_i \sum_{j=1}^n b_i (j) = \sum_{i=1}^n u_i s_i,
$$
where $s_i = s_i (A,B)$.  It follows that
$$
v_{sum} = \sum_{i=1}^n u_i s_i \le u_{max} \sum_{i=1}^n s_i = u_{max} s,
$$
and thus, $v_{sum}/u_{max} \le s$.  Similarly,
$$
v_{sum} = \sum_{i=1}^n u_i s_i \ge u_{min} \sum_{i=1}^n s_i = u_{min} s,
$$
and thus, $v_{sum}/u_{min} \ge s$.
\end{proof}

Recall when $A$ is SFCA over $F$ and $B$ is a right-positive basis for $A$ that the entries for $X(A,B)$ lies in some extension field of $F$.  We let $F[X]$ be the extension of $F$ by the entries in $X (A,B)$.  We now look at the column sums.

\begin{theorem} \label{columns}
Let $A$ be an SFCA over the field $F$, let $B$ be a right-positive basis for $A$, let $E = F[X]$, and let $c_i = \sum_{j=1}^n = b_j (i)$.  Suppose that $E$ is Galois over $F$ and write $G$ for the Galois group of $E$ over $F$.  Suppose that there is an action of $G$ on $B$ so that the eigenvalues of $M(b^\sigma,B)$ are $(b(1)^\sigma, \dots, b(n)^\sigma)$ for all $\sigma \in G$.  (I.e., $(b^\sigma) (j) = (b (j))^\sigma$ for all $j$ and $\sigma$.)  Then $c_i \in F$ for all $i$.
\end{theorem}

\begin{proof}
Observe first that $c_i$ is the sum of the entries of the $i$th column.  Now, the action of an element $\sigma$ of $G$ on $B$ will permute the entries of the $i$th column, and so $\sigma$ just permutes the elements being summed to obtain $c_i$.  In particular, $\sigma$ fixes $c_i$ for all $\sigma \in G$.  This implies that $c_i$ is in the fixed space of $G$ which is $F$.
\end{proof}

\section{Applications}

\subsection{Ordinary Characters}

In Example 1.6 of \cite{chilgen}, Chillag takes $A = \mathbb{Q}[\irr G]$ where $\mathbb{Q}$ is the rational numbers.  He proves that $\irr G$ is a right-positive basis for $\mathbb{Q}[\irr G]$ and that $X(\mathbb{Q}[\irr G],\irr G)$ is the character table for $G$.  Label the characters in $\irr G$ by $\chi_1 = 1_G, \chi_2, \dots, \chi_n$.  Since the entries in $M (\chi,\irr G)$ are integer entries for every character $\chi$, we deduce that $s (\mathbb{Q}[\irr G],\irr G)$ is an integer larger than $n$.

To obtain Solomon's bounds, we take $v_1 = |G|$ and $v_i = 0$ for $i \ne 1$.  We may view $v$ as the regular character of $G$, so we have $u_i = \chi_i (1)$.  Note that $1$ is the minimum of $\{ \chi_1 (1), \dots, \chi_n (1)\}$.  Let $H$ be an abelian subgroup of $G$ of maximal order.  Then the maximum of $\{ \chi_1 (1), \dots, \chi_n (1) \}$ is at most $|G:H|$ (see Problem 2.9 (b) of \cite{text}).  Applying Theorem \ref{bounds}, we have that $|H| = |G|/|G:H| \le s(\mathbb{Q}[\irr G],\irr G) \le |G|$, giving us Solomon's bound.  Notice that for many groups $n$ (the number of irreducible characters) will be larger than the maximal order of abelian subgroup, so we appear to have slightly improved Solomon's lower bound.

We know that the values of the irreducible characters of $G$ are sums of roots of unity, so they all lie in some cyclotomic extension.  I.e., $\mathbb{Q}[\irr G]$ is a subfield of a cyclotomic extension, and thus, $\mathbb{Q}[\irr G]$ is a Galois extension of $\mathbb{Q}$.  It is well-known that the Galois group of $\mathbb{Q}[\irr G]$ acts on $\irr G$ in the appropriate manner (see Problem 2.2 of \cite{text}).  It follows that the sum of each of the columns of the character table for $G$ lie in $\mathbb{Q}$.  But since character values are algebraic integers, these column sums must be integers.

\subsection{Brauer Characters}

In Example 1.8 of \cite{chilgen}, Chillag takes $A = \mathbb{Q}[\ibr G]$.  He proves that $\ibr G$ is a right-positive basis for $\mathbb{Q}[\ibr G]$ and that $X(\mathbb{Q}[\ibr G],\ibr G)$ is the Brauer character table for $G$.  We write $G^o$ for the set of $p$-regular elements of $G$ and if $\chi$ is an ordinary character of $G$, we use $\chi^o$ to denote the restriction of $\chi$ to $G^o$.  We label the Brauer characters in $\ibr G$ by $\phi_1 = 1_{G^o}, \phi_2, \dots, \phi_n$.  By Theorem 2.23 of \cite{navbook}, we know that the entries in $M (\phi,\ibr G)$ are nonnegative integer entries for every Brauer character $\phi$, we deduce that $s (\mathbb{Q}[\ibr G],\ibr G)$ is an integer larger than $n$.

For each $i$, we define $\Phi_i$ to be the projective indecomposable character for $\phi_i$ (see page 25 of \cite{navbook} for the definition of projective indecomposable character).  It is possible to show that $\rho^o = \sum_{i=1}^n \Phi_i (1) \phi_i$ where $\rho$ is the regular character for $G$ (see Problem 2.9 of \cite{navbook}).  Thus, if we take $v$ to be the vector whose first entry is $|G|$ and whose remaining entries are $0$ and $u$ to be the vector whose $i$th entry is $\Phi_i (1)$, then $u$ and $v$ meet the hypotheses of Theorem \ref{bounds}.  Observe that $v_{sum} = |G|$.  By Corollary 2.14 of \cite{navbook}, we have that $\Phi_i (1) \ge |G|_p$, so applying Theorem \ref{bounds}, we have $s (\mathbb{Q}[\ibr G], \ibr G) \le |G|/|G|_p = |G|_{p'}$.

To obtain the lower bound, we need the following lemma.

\begin{lemma}\label{lem:upper bound}
Let $H$ be a $p'$-subgroup of $G$, let $\varphi_i \in \ibr G$ and let $\mu \in \ibr H = \irr H$ be a constituent of $(\varphi_i)_H$. Then $\Phi_i$ is a constituent of $\mu^G$ and so $\Phi_i (1) \leq |G:H|\mu (1)$.  In particular, if $H$ is abelian, then $\Phi_i (1) \leq |G:H|$.
\end{lemma}

\begin{proof}
Since $H$ is a $p'$-subgroup of $G$, $\mu^G$ vanishes off $G^o$,  so $\mu^G = \sum_{j=1}^n a_j \Phi_j$, where $a_j$ an integer by Corollary 2.17 in \cite{navbook}.  We now claim that $a_j \ge 0$ for all $j$ and $a_i > 0$.

For each $j$, we see that $(\phi_j)_H$ is an ordinary character of $H$ and so $[\mu,(\phi_j)_H]$ is a nonnegative integer and we have
$$
|G|_p a_j = |G|_p [\mu^G, \phi_j]^\circ = [\mu^G, \tilde{\phi_j}] = [\mu, (\tilde{\phi_j})_H] = |G|_p [\mu,(\phi_j)_H].
$$
Hence $a_j = [\mu, (\phi_j)_H]$ is a nonnegative integer and if $\mu$ is a constituent of $(\varphi_j)_H,$ then $a_j = [\mu, (\varphi_j)_H] >1$.  Thus $\Phi_i$ is a constituent of $\mu^G$, and so, $\Phi_i (1) \leq |G:H|\mu(1)$.  In particular, if $H$ is an abelian $p'$-subgroup of $G$, then $\Phi_i (1) \le |G:H|$.
\end{proof}

It follows that the entries of $u$ are bounded above by $|G:H|$ where $H$ is an abelian $p'$-subgroup of $G$.  Applying the lower bound from Theorem \ref{bounds}, we have $s (\mathbb{Q}[\ibr G], \ibr G) \ge |G|/|G:H| = |H|$ where $H$ is an abelian $p'$-subgroup of $G$ of maximal order.  In summary, we have ${\rm max} (n, |H|) \le s (\mathbb{Q}[\ibr G], \ibr G) \le |G|_{p'}$ where $n$ is the number of $p$-regular classes of $G$ and $H$ is an abelian $p'$-subgroup of $G$ of maximal order.

On page 43 of \cite{navbook}, Navarro mentions that for an arbitrary group $G$, it is not necessarily true that the Galois Group of $\mathbb{Q}[\ibr G]$ over $\mathbb{Q}$ acts on $\ibr G$.  Thus, Theorem \ref{columns} does not apply.  In fact, it is not the case that the elements of every column of the Brauer character table add up to rational numbers.  For example, using GAP \cite {gap}, one can see that $G = {\rm PSL} (2,16)$ and $p = 2$, $G = {\rm PSL} (2,27)$ and $p = 3$, and $G = {}^2B_2 (32)$ and $p = 2$ all have columns in their Brauer character tables whose sums are not rational numbers.  We will see later when $G$ is $p$-solvable the columns will have integer sums.

\subsection{Projective Indecomposables}

In Example 1.9 of \cite{chilgen}, Chillag takes $A = \mathbb{Q}[{\rm PI} (G)]$ where ${\rm PI} (G) = \{ \Phi_i \}$ is the set of projective indecomposables of $G$.  Let $X ({\rm PI} (G))$ be the ``projective indecomposable table.''  Using Chillag's work, we see that $s ({\rm PI} (G))$ is an integer.  Using Problem 2.9 of \cite{navbook}, we have that $\rho = \sum_{i=1}^n \phi_i (1) \Phi_i$ where as before $\{ \phi_1, \dots, \phi_n \}$ are the irreducible Brauer characters of $G$.  Obviously, $1$ is the minimum of the $\phi_i (1)$, and the $\phi_i (1)$ will be at most the maximum degree of a character in $\irr G$ which we saw before was less than or equal to $|G:A|$ where $A$ is an abelian subgroup of maximal order.  Thus, as in the ordinary character case, we obtain $|A| \le s ({\rm PI} (G)) \le |G|$.  Note that since the Galois group does not act on the irreducible Brauer characters, it also does not act on the projective indecomposables.  Thus, it is not necessarily the case that the column sums in this table are rational numbers, and for example, in the table for $G = {\rm PSL}_2 (16)$ when $p = 2$, there are columns where the sums are not rational.  (This was computed in GAP \cite{gap}.)

\subsection{Partial Characters}

We now show that Chillag's work can be applied to Isaacs' $\pi$-partial characters and the associated $\pi$-projective indecomposable characters.  Isaacs initially defined the $\pi$-partial characters in \cite{pisep} and a different approach to defining the $\pi$-partial characters can be found in \cite{pipart}.  Let $\pi$ be a set of primes, and let $G$ be a $\pi$-separable group.  Let $G^*$ be the set of $\pi$-elements of $G$.  If $\Theta$ is a class function of $G$, then we write $\Theta^*$ to be the restriction of $\Theta$ to $G^*$.  We say that $\phi$ is a {\it $\pi$-partial character} of $G$ if $\phi = \Theta^*$ for some character $\Theta$ of $G$.  We define $\phi$ to be irreducible if it cannot be written as the sum of two other $\pi$-partial characters of $G$.  We write $\Ipi G$ for the set of irreducible $\pi$-partial characters of $G$. In both \cite{pisep} and \cite{pipart}, Isaacs proves that $\Ipi G$ forms a basis for the vector space of complex valued $\pi$-class functions of $G$, and that the $\pi$-partial characters of $G$ are precisely the integer sums of characters in $\Ipi G$ where the coefficients are all nonnegative and at least one is positive.

We also note it is shown in \cite{pisep} that if $G$ is a $p$-solvable group, then $p'$-partial characters of $G$ are precisely the Brauer characters of $G$ and that ${\rm I}_{p'} (G) = \ibr G$.  Hence, any result that we have regarding the $\pi$-partial characters applies to the Brauer characters of $p$-solvable groups.

Following Chillag, we take $A = \mathbb{Q}[\Ipi G]$ and we claim that it is not difficult to see that $\Ipi G$ is a right positive basis for $A$ and that $X (\Ipi G)$ is the $\pi$-partial character table for $G$.  Using Chillag's result, we have that $s (\Ipi G)$ is an integer.  Similar to the Brauer case, we can define the decomposition matrix by $\chi^* = \sum_{\phi \in \Ipi G} d_{\chi \phi} \phi$ where the $d_{\chi \phi}$ are nonnegative integers (see page 92 of \cite{fong}).  Suppose $\phi, \delta \in \Ipi G$.  Then there exist $\chi, \psi \in \irr G$ so that $\chi^* = \phi$ and $\psi^* = \delta$.  We see that $\chi \psi = \sum_{\gamma \in \irr G} a_\gamma \gamma$ for nonnegative integers $a_\gamma$.  We deduce that
$$
\phi \delta = \chi^* \psi^* = \sum a_\gamma \gamma^* = \sum_\gamma a_\gamma \sum_\tau d_{\gamma \tau} \tau,
$$
where $\gamma$ runs over $\irr G$ and $\tau$ runs over $\Ipi G$.  Thus, the multiplicity of $\tau$ in $\phi \delta$ is $\sum_\gamma a_\gamma d_{\gamma \tau}$.  In particular, if $\phi_1, \dots, \phi_n$ are an enumeration of the characters, then $M(\phi_i, \Ipi G)$ has nonnegative integer values.

Following Chapter 10 of \cite{fong}, we define $\Phi_\phi = \sum_{\chi \in \irr G} d_{\chi \phi} \chi$.  Using a proof similar to the proof in the Brauer case, we have that $\rho^* = \sum_{\phi} \Phi_{\phi} (1) \phi$.  On page 105 of \cite{fong}, it is mentioned that $\Phi_\phi$ will be induced from an irreducible character of a Hall $\pi$-subgroup of $G$.  This implies that $|G|_{\pi'} \le \Phi_{\phi} (1)$ for all $\phi \in \Ipi G$.  We also mention that the proof of Lemma \ref{lem:upper bound} applies to the $\pi$-projective indecomposables if $H$ is a $\pi$-subgroup, and so, we have $\Phi_\phi (1) \le |G:H|$ where $H$ is an abelian $\pi$-subgroup of maximal order.  Thus, ${\rm max} (n, |H|) \le s (\Ipi G) \le |G|_{\pi}$ where $n$ is the number of $\pi$-conjugacy classes of $G$ and $H$ is an abelian $\pi$-subgroup of $G$ of maximal order.

Notice also that the $\pi$-projective indecomposable characters will satisfy Chillag's conditions, so if we construct the table of the values of these characters, the sum will be an integer, and as in the case of the projective indecomposables for a single prime $p$, we see that this sum will be between $|A|$ and $|G|$ where $A$ is an abelian subgroup of maximal order.

In Corollary 10.2 of \cite{pisep}, it is proved that there exists a canonical subset $\Bpi G$ of $\irr G$ so that the map $\chi \mapsto \chi^*$ is a bijection from $\Bpi G$ to $\Ipi G$.  Using the definition of $\Bpi G$ one can show that if $\sigma$ is a Galois automorphism for $\mathbb{Q}_{|G|}$ over $\mathbb{Q}$ and $\chi \in \Bpi G$, then $\chi^\sigma \in \Bpi G$.  Using this fact, one can show that if $\phi \in \Ipi G$, then $\phi^\sigma \in \Ipi G$.  In particular, hypotheses of Theorem \ref{columns} are met.  Thus, the sums of the columns of the tables for the $\pi$-partial characters and the $\pi$-projective indecomposable characters lie in $\mathbb{Q}$.  Again since the entries are known to be algebraic integers, we conclude the sums of the columns in these tables are integers.  Note that this recovers the fact that the columns of the Brauer table for $p$-solvable groups are integers that was proved by Chen and Du in \cite{chendu}.


\section*{Acknowledgments}
The first author would like to thank the support of China Scholarship Council,
Department of Mathematical Sciences of Kent State University for its hospitality,
Funds of Henan University of Technology (2014JCYJ14, 2016JJSB074, 26510009),
Project of Education Department of Henan Province (17A110004), 
Projects of Zhengzhou Municipal Bureau of Science and Technology (20150249, 20140970),
and the NSFC (11571129).


\end{document}